\documentclass[12pt]{amsart}
\usepackage{amsmath,amssymb,amsthm,color}
\usepackage[numbers]{natbib}
\usepackage{url}

\usepackage{slashed}
\usepackage{mathtools,dsfont,amsfonts,bbm}

\newtheorem{lem}{Lemma}[section]
\newtheorem{prop}{Proposition}[section]
\newtheorem{cor}{Corollary}[section]
\newtheorem{thm}{Theorem}[section]
\newtheorem*{conj}{Conjecture}
\theoremstyle{remark}

\numberwithin{equation}{section}

\usepackage{esint}
\mathtoolsset{showonlyrefs}

\newcommand{\R}{\mathbb R}
\newcommand{\F}{\mathcal F}

\begin{document}
\title[On scattering-critical NLS]{On 2D Scattering for a Critical Nonlinear Schr\"odinger Flow}
\author[Vladimir Georgiev]{Vladimir Georgiev$^*$}
\thanks{$^*$Corresponding author.}
\address{Dipartimento di Matematica, Universit\`a di Pisa, Largo B. Pontecorvo 5, 56100 Pisa, Italy;\\ Faculty of Science and Engineering, Waseda University, 3-4-1 Okubo, Shinjuku-ku, Tokyo 169-8555, Japan;\\ Institute of Mathematics and Informatics, Bulgarian Academy of Sciences, Acad. Georgi Bonchev Str., Block 8, 1113 Sofia, Bulgaria}
\author{Tohru Ozawa}
\address{Department of Applied Physics, Waseda University, 3-4-1 Okubo, Shinjuku-ku, Tokyo 169-8555, Japan}
\date{}

\begin{abstract}
We study modified scattering for the two-dimensional defocusing nonlinear Schr\"odinger equation (NLS) with the gauge-invariant, scattering-critical nonlinearity $|u|u$. We prove that the remodulated interaction representation has a local $L^2$ limit for general data in $\Sigma=H^1\cap\F^{-1}H^1$. For radial data, we prove convergence in $L^p(\R^2)$ for every $2<p<\infty$.
\end{abstract}
\maketitle

\textbf{Subject Class:} 35Q55, 35P25, 35B40.

\section{Statement of the problem and main result}

The scattering behavior of solutions to the nonlinear Schr\"odinger equation (NLS) is closely related to the existence of the limits
\begin{equation}\label{eqWOPER1}
\lim_{t\to\pm\infty}U(-t)u(t),
\end{equation}
where
\begin{equation}\label{eqNLS}
i\partial_tu+\frac12\Delta u=u|u|^{\alpha},\qquad u(0)=u_0\in H^1(\R^n),
\end{equation}
and \(U(t)=e^{it\Delta/2}\) denotes the free Schr\"odinger group. In the mass-supercritical, energy-subcritical regimes, standard functional settings for scattering and wave operators are discussed in \cite{C03}. In the mass-subcritical regime, a natural functional framework is provided by the Fourier-invariant weighted Sobolev space
\[ \Sigma=H^1(\R^n)\cap\F^{-1}H^1(\R^n). \]
When \(0<\alpha\le 2/n\), the nonlinear interaction is long-range, and the usual scattering profile \(U(-t)u(t)\) is not expected to converge without a suitable phase correction. This leads naturally to modified wave operators and modified scattering.

In this paper, we focus on the two-dimensional defocusing nonlinear Schr\"odinger equation with the critical quadratic nonlinearity, corresponding to \(\alpha=1\) and \(n=2\):
\begin{equation}\label{NLSmm}
i\partial_tu+\frac12\Delta u=u|u|, \qquad (t,x)\in\R\times\R^2.
\end{equation}
For general, not necessarily small, initial data, a principal difficulty is to prove the existence of a strong \(L^2\) limit of the modified profile as \(t\to\infty\). Although the \(L^2\) norm of the profile is controlled by mass conservation, this bound alone does not provide the compactness needed to establish strong convergence in \(L^2(\R^2)\). Nevertheless, in the radial setting, one may expect the modified profile to converge strongly in \(L^p(\R^2)\) for every \(2<p<\infty\), with a limiting state belonging to the corresponding \(L^p\) spaces. Thus, the natural large-data modified-scattering statement is formulated in \(L^p\), rather than through global strong convergence in \(L^2\). This point of view is motivated by the linear \(L^p\) wave-operator theory of Yajima \cite{Y99}. In the nonlinear setting, however, the effective potential depends both on the solution and on time, which introduces additional difficulties. For \(t\ge1\), we define the phase as a function of the full spatial variable \(x\in\R^2\) by
\begin{equation}\label{eq.phd}
(\Phi_t(u))(x):=\int_1^t \left|u\left(\tau,\frac{\tau}{t}x\right)\right|\,d\tau.
\end{equation}
We then define the phase-corrected solution \(w\) and the corresponding modified profile \(v\) by
\begin{equation}\label{signconv}
w(t):=e^{i\Phi_t(u)}u(t),\qquad v(t):=U(-t)w(t).
\end{equation}
This choice of phase and sign convention agrees with the results of \cite{GO25}. These considerations lead to the following conjectural nonlinear analogue of Yajima's \(L^p\) wave-operator theory.
\begin{conj}[Radial nonlinear Yajima-type conjecture] Let \(u_0\in\Sigma\) be radial, and let \(u\) be the corresponding global solution to \eqref{NLSmm}. Then there exists
\[ u_+\in L^2(\R^2)\cap\bigcap_{2<p<\infty}L^p(\R^2) \]
such that, for every \(2<p<\infty\),
\[ \|U(-t)e^{i\Phi_t(u)}u(t)-u_+\|_{L^p}\to0 \qquad\text{as }t\to\infty. \]
\end{conj}
The conjecture does not assert strong convergence of the modified profile in \(L^2(\R^2)\). Rather, it asserts the existence of an \(L^2\) limiting state \(u_+\) and strong convergence to this state in every \(L^p(\R^2)\) with \(2<p<\infty\).

\newpage

\subsection{Main results}
\begin{thm}\label{MT0}
Let \(u\) be a global solution to \eqref{NLSmm} with
\[ u_0\in \mathcal F^{-1}H^1(\R^2). \]
Then there exists a unique \(u_+\in L^2(\R^2)\) such that, for every \(R>0\),
\begin{equation}\label{mwoe60m}
\lim_{t\to\infty} \left\|U(-t)e^{i\Phi_t(u)}u(t)-u_+\right\|_{L^2(|x|<R)}=0,
\end{equation}
and, for every \(s>1\),
\begin{equation}
\lim_{t\to\infty} \left\|\langle x \rangle^{-s} \left( U(-t)e^{i\Phi_t(u)}u(t)-u_+\right)\right\|_{L^2}=0,
\end{equation}
where $\langle x \rangle = (1+|x|^2)^{1/2}.$
\end{thm}
\begin{thm}\label{MT1} Let \(u\) be a radial global solution to \eqref{NLSmm} with
\[ u_0\in H^1_{\mathrm{rad}}(\R^2)\cap \mathcal F^{-1}H^1_{\mathrm{rad}}(\R^2). \]
Then there exists a unique
\[ u_+\in L^2(\R^2)\cap\bigcap_{2<p<\infty}L^p(\R^2) \]
such that, for every \(2<p<\infty\),
\begin{equation}\label{mwoe60}
\lim_{t\to\infty} \left\|U(-t)e^{i\Phi_t(u)}u(t)-u_+\right\|_{L^p}=0.
\end{equation}
\end{thm}

We briefly review some results on modified scattering for the NLS in two dimensions. Most of these works concern small-data modified scattering and the construction of modified wave operators, namely solutions with prescribed asymptotic data at \( \pm\infty \). The study of modified wave operators for NLS began in 1991 with the work \cite{O91}, which considered the one-dimensional critical NLS with small initial data.
The cases \(n=1,2,3\) with \(\alpha=2/n\), under a small-data assumption, were subsequently investigated in \cite{GO93}. Small-data asymptotics in \(\Sigma\) were established in \cite{HN98AJM}, while questions concerning the domain and range of modified wave operators under weaker regularity assumptions were addressed in \cite{HN06}.
The existence of weak \(H^1\) limits for modified wave operators was studied in \cite{H95} for Hartree-type NLS when \(n\geq 3\), and for NLS in the two-dimensional critical case \(n=2\) with quadratic nonlinearity \((\alpha=1)\). The results in \cite{H95} show the existence of a weak \(H^1\) limit of \begin{equation} M(t)U(-t)e^{i\Phi_t(u)}u(t), \end{equation} where \(M(t)\) denotes multiplication by \(e^{i|x|^2/(2t)}\).
More recently, \cite{VH25} established small-data modified scattering results for \(n=2,3\) in the presence of combined local and nonlocal nonlinearities. In our previous work \cite{GO25}, we studied modified wave operators with large initial data at \(t=0\) for \(n=1,2\) and \(\alpha=1\). Adapting the method of \cite{H95}, we proved the existence of a strong limit of \begin{equation} U(-t)e^{i\Phi_t(u)}u(t) \end{equation} in the space \(\mathcal F^{-1}(\dot B^{-1}_{2,\infty})\).

\section{Proof of the main results}
\subsection{Proof of Theorem \ref{MT0}}
We recall \cite[Theorem 1.5]{GO25}.
\begin{thm}\label{th5}
Let $n=2$. For any $u_0\in\F^{-1}(H^1)$ there exists a unique $u_+\in L^2$ such that
\begin{equation}\label{wp1}
 \|U(-t)e^{i\Phi_t(u)}u(t)-u_+\|_{\F^{-1}(\dot B^{-1}_{2,\infty})}\to0.
\end{equation}
\end{thm}
Let
\[
 g_t:=U(-t)e^{i\Phi_t(u)}u(t)-u_+,
 \qquad A_j:=\{2^{j-1}<|x|\le2^j\}.
\]
The spatial dyadic characterization from Appendix A gives
\begin{equation}\label{annconv}
 M_t:=\sup_{j\in\mathbb Z}2^{-j}\|g_t\|_{L^2(A_j)}\to0.
\end{equation}
Fix $R>0$ and choose $j_0$ so that $2^{j_0-1}<R\le2^{j_0}$. Then
\begin{equation}\label{both1}
    \begin{aligned}
        \|g_t\|_{L^2(|x|<R)}^2
 &\le\sum_{j=-\infty}^{j_0}\|g_t\|_{L^2(A_j)}^2\\
 &\le M_t^2\sum_{j=-\infty}^{j_0}2^{2j}
 =\frac43 2^{2j_0}M_t^2\lesssim R^2M_t^2.
    \end{aligned}
\end{equation}

Moreover,
\begin{equation}\label{both2}
\begin{aligned}
     &\|\langle x \rangle^{-s}g_t\|_{L^2(|x|>R)}^2  \leq \sum_{j=j_0-1}^\infty \int_{A_j} \langle x \rangle^{-2s} |g_t(x)|^2 dx\\
     &\leq \sum_{j=j_0-1}^\infty 2^{-2js}\|g_t\|_{L^2(A_j)}^2 \\
    & \leq \left( \sup_{j \in \mathbb Z} 2^{-j} \|g_t\|_{L^2(A_j)} \right)^2 \left( \sum_{j=j_0-1}^\infty 2^{2j-2js} \right) \lesssim R^{2(1-s)} M_t^2.
\end{aligned}
\end{equation}

Combining \eqref{both1} and \eqref{both2}, we obtain the desired weighted convergence.

It remains only to note uniqueness in the class \(L^2(\R^2)\). Suppose that
\(u_+,\widetilde u_+\in L^2(\R^2)\) both satisfy the conclusion of
Theorem \ref{MT0}. Then, for every \(R>0\),
\begin{equation}
    \begin{aligned}
        \|u_+-\widetilde u_+\|_{L^2(|x|<R)}&\\
 \le
 \|U(-t)e^{i\Phi_t(u)}u(t)-u_+\|_{L^2(|x|<R)}&
 +
 \|U(-t)e^{i\Phi_t(u)}u(t)-\widetilde u_+\|_{L^2(|x|<R)}.
    \end{aligned}
\end{equation}
Letting \(t\to\infty\), we obtain
\[
 \|u_+-\widetilde u_+\|_{L^2(|x|<R)}=0
\]
for every \(R>0\). Hence \(u_+=\widetilde u_+\) a.e. on \(\R^2\).
This completes the proof.

\subsection{Preliminary estimates for Theorem \ref{MT1}}
From \eqref{signconv}  and the operator identity $(1-\Delta)^{-1}(1-\Delta)=I$,
\begin{equation}\label{decomp}
\begin{aligned}
 v(t)&=f_1(t)+f_2(t)+f_3(t),\\
 f_1(t)&=-(1-\Delta)^{-1}\nabla\cdot U(-t)[(\nabla e^{i\Phi_t(u)})u(t)],\\
 f_2(t)&=-(1-\Delta)^{-1}\nabla\cdot U(-t)[e^{i\Phi_t(u)}\nabla u(t)],\\
 f_3(t)&=(1-\Delta)^{-1}U(-t)[e^{i\Phi_t(u)}u(t)].
\end{aligned}
\end{equation}

\begin{lem}\label{lemf}
For $4\le q<\infty$, $t>1$
\[
 f_1(t)\in W^{1,q},\quad f_2(t)\in H^1,\quad f_3(t)\in W^{1,q}\cap H^2,
\]
and,
\begin{equation}\label{estf1f3}
\begin{aligned}
 &  \|f_1(t)\|_{W^{1,q}}\lesssim t^{-1/2+2/q},\\
& \|f_2(t)\|_{H^1}\lesssim1,\\
&  \|f_3(t)\|_{W^{1,q}}+\|f_3(t)\|_{H^2}\lesssim1.
\end{aligned}
\end{equation}
\end{lem}

\begin{proof}
The estimate for \(f_2\) follows from the \(L^2\)-boundedness of the
multiplier
\[
 (1-\Delta)^{-1}\nabla\cdot
\]
from \(L^2\) to \(H^1\), the unitarity of \(U(t)\) on \(L^2\), and the
uniform \(H^1\) bound for \(u\). Indeed,
\[
 \|f_2(t)\|_{H^1}
 \lesssim
 \|e^{i\Phi_t(u)}\nabla u(t)\|_2
 =
 \|\nabla u(t)\|_2
 \lesssim 1.
\]

Similarly,
\[
 \|f_3(t)\|_{H^2}
 \lesssim
 \|e^{i\Phi_t(u)}u(t)\|_2
 =
 \|u(t)\|_2
 =
 \|u_0\|_2.
\]
Since \(H^2(\R^2)\hookrightarrow W^{1,q}(\R^2)\) for every finite \(q\),
we also have
\[
 \|f_3(t)\|_{W^{1,q}}\lesssim 1.
\]

It remains to estimate \(f_1\). The operator
\[
 (1-\Delta)^{-1}\nabla\cdot
\]
maps \(L^q\) boundedly into \(W^{1,q}\) for \(1<q<\infty\). Hence, by
the dispersive estimate in dimension two,
\[
 \|f_1(t)\|_{W^{1,q}}
 \lesssim
 \|U(-t)[(\nabla e^{i\Phi_t(u)})u(t)]\|_{L^q}
 \lesssim
 t^{-1+2/q}
 \|(\nabla e^{i\Phi_t(u)})u(t)\|_{L^{q'}}.
\]
Choose
\[
 \mu=\frac{2q}{q-2},
\]
so that
\[
 \frac1{q'}=\frac12+\frac1\mu.
\]
Since
\[
 \nabla e^{i\Phi_t(u)}
 =
 i e^{i\Phi_t(u)}\nabla\Phi_t(u),
\]
Hölder's inequality gives
\[
 \|(\nabla e^{i\Phi_t(u)})u(t)\|_{L^{q'}}
 \le
 \|\nabla\Phi_t(u)\|_{L^2}\|u(t)\|_{L^\mu}.
\]
For \(q\ge4\), we have \(2\le\mu\le4\), and therefore the uniform
\(H^1\) bound implies
\[
 \|u(t)\|_{L^\mu}\lesssim1.
\]

By the chain rule,
\[
 \nabla_x\Phi_t(u)(x)
 =
 \int_1^t
 \frac{\tau}{t}
 \nabla |u|\left(\tau,\frac{\tau}{t}x\right)\,d\tau.
\]
In two dimensions, the scaling factors cancel:
\[
 \left\|
 \frac{\tau}{t}
 \nabla |u|\left(\tau,\frac{\tau}{t}x\right)
 \right\|_{L_x^2}
 =
 \|\nabla |u(\tau)|\|_2.
\]
Using the diamagnetic inequality and the pseudoconformal bound from
Appendix B, we obtain
\[
 \|\nabla |u(\tau)|\|_2
 \le
 \left\|\nabla\left(e^{-i|x|^2/(2\tau)}u(\tau)\right)\right\|_2
 =
 \tau^{-1}\|J(\tau)u(\tau)\|_2
 \lesssim
 \tau^{-1/2}.
\]
Here \(J(t)=x+it\nabla\), so that
\[ \left\|\nabla\left(e^{-i|x|^2/(2t)}u(t)\right)\right\|_2 = t^{-1}\|J(t)u(t)\|_2 . \]
Consequently,
\[
 \|\nabla\Phi_t(u)\|_2
 \lesssim
 \int_1^t \tau^{-1/2}\,d\tau
 \lesssim t^{1/2}.
\]
Combining the preceding estimates gives
\[
 \|f_1(t)\|_{W^{1,q}}
 \lesssim
 t^{-1+2/q}t^{1/2}
 =
 t^{-1/2+2/q}.
\]
This proves \eqref{estf1f3}.
\end{proof}

\subsection{Proof of Theorem \ref{MT1}}

We first prove the following lemma..
\begin{lem}\label{lstep2}
If $u(t)$ is a radial function in $\mathbb R^2$, then $f_1,f_2,f_3$ in \eqref{decomp} are radial.
\end{lem}
\begin{proof}

Since $u(t)$ is radial and the phase $\Phi_t(u)$ is radial, the function \[ w:=e^{i\Phi_t(u)}u(t) \] defined in \eqref{signconv} is radial. By the definition of $f_3$ in \eqref{decomp}, this also implies that $f_3$ is radial.

Moreover,
\[
 \nabla e^{i\Phi_t}=i e^{i\Phi_t}\frac{x}{|x|}\partial_r\Phi_t, \ \ \partial_r = \frac{x}{|x|} \cdot \nabla,
\]
so $(\nabla e^{i\Phi_t})u$ is a radial vector field, namely a vector field of the form \[ \frac{x}{|x|}a(|x|) \] for some scalar radial function $a$. Equivalently, it belongs to the degree-one spherical-harmonic sector. By Appendix C, the propagator $U(-t)$ preserves this sector.  Moreover, the radial resolvent appearing in the definition of $f_1$ also preserves spherical-harmonic sectors. Therefore, after applying the radial resolvent, we still obtain a degree-one radial vector field. Taking the divergence of such a vector field gives a radial scalar function, since
\[ \operatorname{div}\!\left(\frac{x}{|x|}a(|x|)\right) = a'(|x|)+\frac{1}{|x|}a(|x|), x \in \mathbb R^2 \]
Thus $f_1$ is radial. Finally, since $v$, $f_1$, and $f_3$ are radial, the identity \[ f_2=v-f_1-f_3 \] implies that $f_2$ is radial as well.
\end{proof}

We are now in a position to apply the Strauss--Lions estimate \cite{St77,Li82,CO09}, so we can write
\[
 |f(x)|\lesssim |x|^{-1/q}\|f\|_{W^{1,q}_{\rm rad}},\qquad 1\le q<\infty.
\]
Consequently, if $\rho>2q$ and $0<\delta<1/q-2/\rho$,
\begin{equation}\label{eq.rhoq}
 \|\langle x\rangle^\delta f\|_{L^\rho}\lesssim\|f\|_{W^{1,q}_{\rm rad}}.
\end{equation}
Indeed, the local part follows from Sobolev embedding, while
\[
 \||x|^\delta f\|_{L^\rho(|x|>1)}^\rho
 \lesssim\|f\|_{W^{1,q}}^\rho\int_1^\infty r^{\rho\delta-\rho/q+1}dr,
\]
which converges precisely under the stated condition. Applying this with $q=4$ to $f_1,f_3$, and with $q=2$ to $f_2$, yields
\begin{equation}\label{weighted}
 \sup_{t\ge1}\||x|^\delta v(t)\|_{L^\rho(|x|>1)}<\infty
\end{equation}
whenever $\rho>8$ and $\delta+2/\rho<1/4$.

On the unit ball, the decomposition \eqref{decomp}, the uniform bounds
\[
\sup_{t\ge1}\bigl(
\|f_1(t)\|_{W^{1,4}}
+\|f_2(t)\|_{H^1}
+\|f_3(t)\|_{H^2}
\bigr)<\infty,
\]
and the two-dimensional Sobolev embeddings
\[
W^{1,4}\hookrightarrow L^\infty,\qquad
H^1\hookrightarrow L^\rho\quad(2\le\rho<\infty),\qquad
H^2\hookrightarrow L^\infty
\]
give
\begin{equation}\label{localrho}
 \sup_{t\ge1}\|v(t)\|_{L^\rho(|x|<1)}<\infty
\end{equation}
for every \(2\le\rho<\infty\). In particular, this holds for every
finite \(\rho>8\).

Next, we turn to the exterior interpolation. Applying Corollary
\ref{trealint} to \(f=v(t)-u_+\), we obtain
\begin{equation}\label{wpv1}
 \||x|^{-\theta}(v(t)-u_+)\|_2
 \lesssim
 \|v(t)-u_+\|_{\F^{-1}\dot B^{-1}_{2,\infty}}^\theta
 \|v(t)-u_+\|_2^{1-\theta},
\end{equation}
where \(\theta\in(0,1)\). Since \(v(t)\to u_+\) in
\(\F^{-1}\dot B^{-1}_{2,\infty}\) and \(\|v(t)-u_+\|_2\) is uniformly
bounded, it follows that
\[
 \||x|^{-\theta}(v(t)-u_+)\|_{L^2(|x|>1)}\to0.
\]

 Fix $p>2$ and choose
\[
 \rho>\max\{8,p\},\qquad
 \sigma=\frac{1/p-1/\rho}{1/2-1/\rho}\in(0,1).
\]
Choose $\theta>0$ such that
\[
 \delta:=\frac{\sigma\theta}{1-\sigma},\qquad
 \delta+\frac2\rho<\frac14.
\]
Weighted H\"older interpolation gives
\[
 \|f\|_{L^p(|x|>1)}
 \leq \||x|^{-\theta}f\|_2^\sigma
 \||x|^\delta f\|_\rho^{1-\sigma}.
\]

By \eqref{weighted}, the family
\[
\{|x|^\delta v(t):t\ge1\}
\]
is bounded in the reflexive space \(L^\rho(|x|>1)\). Hence, for any
sequence \(t_k\to\infty\), there is a subsequence, still denoted \(t_k\),
and some \(G\in L^\rho(|x|>1)\) such that
\[
|x|^\delta v(t_k)\rightharpoonup G
\quad\text{weakly in }L^\rho(|x|>1).
\]
On every bounded annulus \(A_R=\{x \in \mathbb R^2: 1<|x|<R\} \), Theorem \ref{MT0} gives
\[
v(t_k)\to u_+
\quad\text{strongly in }L^2(A_R).
\]
It follows that
\[
G=|x|^\delta u_+
\quad\text{a.e. on } A_R.
\]
Since \(R>1\) is arbitrary,
\[
|x|^\delta u_+\in L^\rho(|x|>1),
\]
and weak lower semicontinuity gives
\[
\||x|^\delta u_+\|_{L^\rho(|x|>1)}
\le
\liminf_{k\to\infty}
\||x|^\delta v(t_k)\|_{L^\rho(|x|>1)}
\lesssim1.
\]
Consequently,
\[
\sup_{t\ge1}
\||x|^\delta(v(t)-u_+)\|_{L^\rho(|x|>1)}
<\infty.
\]
Combining this with
\[
\||x|^{-\theta}(v(t)-u_+)\|_2\to0
\]
and weighted Hölder interpolation yields
\[
\|v(t)-u_+\|_{L^p(|x|>1)}\to0.
\]

It remains to treat the origin. By Theorem \ref{MT0},
\[
 \|v(t)-u_+\|_{L^2(|x|<1)}\to0.
\]
By \eqref{localrho} and weak lower semicontinuity, $u_+\in L^\rho(|x|<1)$ and $v(t)-u_+$ is uniformly bounded there in $L^\rho$. If
\[
 \frac1p=\frac\alpha2+\frac{1-\alpha}{\rho},
\]
then
\[
 \|v(t)-u_+\|_{L^p(|x|<1)}
 \le\|v(t)-u_+\|_{L^2(|x|<1)}^\alpha
 \|v(t)-u_+\|_{L^\rho(|x|<1)}^{1-\alpha}\to0.
\]
This completes the proof.

\appendix
\section{Interpolation in $\F^{-1}(\dot B^s_{2,r})$}

Let $\varphi\in C_c^\infty((1/2,2))$ be chosen so that the functions
\begin{equation}\label{eq:dyadic-cutoff}
 \varphi_j(\xi):=\varphi(2^{-j}|\xi|),
 \qquad j\in\mathbb Z,
\end{equation}
form a homogeneous Paley--Littlewood partition of unity, namely
\begin{equation}\label{eq:partition-unity}
 \sum_{j\in\mathbb Z}\varphi_j(\xi)=1,
 \qquad \xi\neq 0.
\end{equation}
For $s\in\mathbb R$ and $1\le r\le\infty$, we define the Fourier-side homogeneous Besov space
$\F^{-1}\dot B^s_{2,r}$ by
\begin{equation}\label{FB1}
 \|G\|_{\F^{-1}\dot B^s_{2,r}}
 :=
 \left\|
   2^{js}\|\varphi_jG\|_{L^2_\xi}
 \right\|_{\ell^r_{j\in\mathbb Z}}.
\end{equation}
Here $G$ is understood as a tempered distribution and the products
$\varphi_jG$ are taken in the distributional sense. Thus
$G\in\F^{-1}\dot B^s_{2,r}$ if $\varphi_jG\in L^2(\R^2)$ for every
$j\in\mathbb Z$ and the right-hand side of \eqref{FB1} is finite.

For $r=\infty$, the definition is interpreted as
\[
 \|G\|_{\F^{-1}\dot B^s_{2,\infty}}
 =
 \sup_{j\in\mathbb Z}
 2^{js}\|\varphi_jG\|_{L^2_\xi}.
\]
As usual for homogeneous Besov spaces, this defines a norm modulo
distributions supported at the origin in the Fourier variable.

For \(s\in\R\) and \(1\le r\le\infty\), we use the notation $\dot \ell_r^s$ for discrete  dyadic weighted  Lebesgue space with norm
\begin{equation}
  \|a\|_{\dot \ell_r^s} =
 \left\{\begin{aligned}
     &\left( \sum_{j\in \mathbb Z} |2^{js} a_j|^r \right)^{1/r}  \mbox{if $1 \leq r < \infty$,} \\
     &\|2^{js} a_j\|_{\ell^\infty} \ \ \mbox{if $r=\infty$}
  \end{aligned}\right.
\end{equation}
for any $a = (a_j)_{j \in \mathbb Z} \in \dot \ell_r^s. $ In the case of sequences with values in a Banach space $\mathcal B$ we have the corresponding space  $\dot \ell_r^s (\mathcal B).$ In the particular case $\mathcal B = L^2$ for any $G \in L^2_{loc}(\mathbb R^2 \setminus \{0\}) $ we can define the dyadic annuli
\[
 A_j:=\{x\in\R^2:2^j\le |x|<2^{j+1}\},
 \qquad j\in\mathbb Z.
\]
and define
\[
 \|G\|_{\dot \ell_r^s(L^2(\mathbb R^2))}
 :=
 \left\| \|G\|_{L^2(A_j)} \right\|_{\dot \ell_r^s}
 \]
Then we have
\[
 \dot \ell_r^s(L^2(\mathbb R^2))=\F^{-1}\dot B^s_{2,r} (\mathbb R^2)
\]
with equivalent norms, i.e.
\begin{equation}\label{equivBesov}
    \|G\|_{\dot \ell_r^s(L^2)}
 \sim
 \|G\|_{\F^{-1}\dot B^s_{2,r}}.
\end{equation}

This is true because the Littlewood multipliers $\varphi_j$ are supported on finitely overlapping dyadic annuli and there exists an integer $N \geq 1$ so that for any $j \in \mathbb Z$ we have
\begin{equation}
    \mathds{1}_{A_j}(\xi) \leq \sum_{|j-k|\leq N} \varphi_k(\xi), \qquad   \varphi_j(\xi) \leq \sum_{|j-k|\leq N}\mathds{1}_{A_k} (\xi).
\end{equation}

Before stating the main interpolation result of this section we prove the following elementary discrete Hardy inequalities.

\begin{lem}\label{lem:discrete-hardy}
Let \(\alpha>0\) and \(1\le p,q\le\infty\). Then
\begin{equation}\label{eq:hardy-tail}
 \left\|
  2^{\alpha n}
  \left\|a_j\right\|_{\ell^q_{j>n}}
 \right\|_{\ell^p_n}
 \sim
 \left\|2^{\alpha j}a_j\right\|_{\ell^p_j},
\end{equation}
and
\begin{equation}\label{eq:hardy-head}
 \left\|
  2^{-\alpha n}
  \left\|a_j\right\|_{\ell^q_{j\le n}}
 \right\|_{\ell^p_n}
 \sim
 \left\|2^{-\alpha j}a_j\right\|_{\ell^p_j}.
\end{equation}
The constants depend only on \(\alpha,p,q\).
\end{lem}

\begin{proof}
We prove \eqref{eq:hardy-tail}; the proof of \eqref{eq:hardy-head} is identical after reversing the order of the indices.

Put
\[
 b_j:=2^{\alpha j}|a_j|.
\]
Since \(q\ge1\),
\[
 \left\|a_j\right\|_{\ell^q_{j>n}}
 \le
 \sum_{j>n}|a_j|
 =
 \sum_{j>n}2^{-\alpha j}b_j.
\]
Hence
\[
 2^{\alpha n}
 \left\|a_j\right\|_{\ell^q_{j>n}}
 \le
 \sum_{j>n}2^{-\alpha(j-n)}b_j.
\]
The right-hand side is the convolution of \(b\) with the summable kernel
\[
 k_m=
 \begin{cases}
 2^{-\alpha m},& m\ge1,\\
 0,& m\le0.
 \end{cases}
\]
Therefore, by Young's inequality on \(\mathbb Z\),
\[
 \left\|
  2^{\alpha n}
  \left\|a_j\right\|_{\ell^q_{j>n}}
 \right\|_{\ell^p_n}
 \le
 C
 \|b\|_{\ell^p}
 =
 C
 \left\|2^{\alpha j}a_j\right\|_{\ell^p_j}.
\]

Conversely, for every \(j\),
\[
 \left\|a_k\right\|_{\ell^q_{k>j-1}}
 \ge |a_j|.
\]
Thus
\[
 2^{\alpha(j-1)}
 \left\|a_k\right\|_{\ell^q_{k>j-1}}
 \ge
 2^{-\alpha}2^{\alpha j}|a_j|.
\]
Taking the \(\ell^p\)-norm in \(j\) gives the reverse inequality. The case
\(p=\infty\) is obtained in the same way, replacing sums by suprema.
\end{proof}

The statement of the Proposition below  is Theorem 5.6.1 in \cite{BL76} (for the case of homogeneous Besov spaces see \cite{OT25} and Theorem 2.6 in \cite{G24}). For completeness we give a self-contained proof.
\begin{prop}\label{prop:weighted-sequence-interpolation}
Let $\mathcal B$ be a Banach space, let $s\neq0$, let
$0<\theta<1$, and let
\[
 1\le q_1,q_2,p\le\infty.
\]
Then
\begin{equation}\label{eq:weighted-sequence-interpolation}
 \bigl(
   \dot\ell^0_{q_1}(\mathcal B),
   \dot\ell^s_{q_2}(\mathcal B)
 \bigr)_{\theta,p}
 =
 \dot\ell^{\theta s}_{p}(\mathcal B)
\end{equation}
with equivalent norms. More precisely,
\[
 \|a\|_{
   (\dot\ell^0_{q_1}(\mathcal B),
    \dot\ell^s_{q_2}(\mathcal B))_{\theta,p}
 }
 \sim
 \left\|
   2^{\theta sj}\|a_j\|_{\mathcal B}
 \right\|_{\ell^p_j},
\]
where the constants of equivalence depend only on
$s,\theta,p,q_1$, and $q_2$.
\end{prop}

\begin{proof}
We first prove the result for $s>0$. The case $s<0$ will be
deduced at the end by reversing the indices.

Set
\[
 A:=\dot\ell^0_{q_1}(\mathcal B)
   =\ell^{q_1}(\mathcal B),
 \qquad
 B:=\dot\ell^s_{q_2}(\mathcal B).
\]
Thus
\[
 \|a\|_A
 =
 \left\|\|a_j\|_{\mathcal B}\right\|_{\ell^{q_1}_j},
 \qquad
 \|a\|_B
 =
 \left\|2^{sj}\|a_j\|_{\mathcal B}\right\|_{\ell^{q_2}_j}.
\]

For $t>0$, the $K$-functional associated with the compatible
couple $(A,B)$ is
\[
 K(t,a;A,B)
 :=
 \inf_{a=b+c}
 \bigl(
   \|b\|_A+t\|c\|_B
 \bigr).
\]

We shall estimate the $K$-functional at the dyadic values
\[
 t_n:=2^{-ns},
 \qquad n\in\mathbb Z.
\]

We first record the elementary weighted sequence estimate that will
be used in the proof. If $\alpha>0$ and
$1\le p,q\le\infty$, then
\begin{equation}\label{eq:geometric-weight-embedding}
 \left\|
   2^{-\alpha m}a_m
 \right\|_{\ell^p_{m\ge0}}
 \le
 C_{\alpha,p,q}
 \|a_m\|_{\ell^q_{m\ge0}}.
\end{equation}
Indeed, if $p\ge q$, then
$\ell^q\hookrightarrow\ell^p$, and hence
\[
 \left\|2^{-\alpha m}a_m\right\|_{\ell^p}
 \le
 \left\|2^{-\alpha m}a_m\right\|_{\ell^q}
 \le
 \|a_m\|_{\ell^q}.
\]
If $p<q$, choose $r\in[1,\infty)$ such that
\[
 \frac1p=\frac1q+\frac1r.
\]
Hölder's inequality gives
\[
 \left\|2^{-\alpha m}a_m\right\|_{\ell^p}
 \le
 \left\|2^{-\alpha m}\right\|_{\ell^r}
 \|a_m\|_{\ell^q},
\]
and the geometric sequence $(2^{-\alpha m})_{m\ge0}$ belongs to
$\ell^r$. The endpoint cases involving $p=\infty$ or $q=\infty$
follow from the same argument, with the usual modifications.

Let
\[
 a\in A+B.
\]
Thus there exist $b\in A$ and $c\in B$ such that
\[
 a=b+c.
\]
For $n\in\mathbb Z$, define
\[
 H_n(a)
 :=
 \left\|\|a_j\|_{\mathcal B}\right\|_{\ell^{q_1}_{j>n}},
\]
and
\[
 L_n(a)
 :=
 2^{-ns}
 \left\|
   2^{sj}\|a_j\|_{\mathcal B}
 \right\|_{\ell^{q_2}_{j\le n}}.
\]
These quantities are finite for every $a\in A+B$. Indeed, if
$a=b+c$ with $b\in A$ and $c\in B$, then
\[
 H_n(a)
 \le
 \left\|\|b_j\|_{\mathcal B}\right\|_{\ell^{q_1}_{j>n}}
 +
 \left\|\|c_j\|_{\mathcal B}\right\|_{\ell^{q_1}_{j>n}}.
\]
The first term is bounded by $\|b\|_A$, while the estimate \eqref{eq:geometric-weight-embedding} implies
\[
 \left\|\|c_j\|_{\mathcal B}\right\|_{\ell^{q_1}_{j>n}}\le
  C2^{-ns} \left\|2^{sj}\|c_j\|_{\mathcal B}\right\|_{\ell^{q_2}_{j>n}}
 \le
 C2^{-ns}\|c\|_B.
\]
so we have $H_n(a)<\infty$.

Similarly,
\[
 L_n(a)
 \le
 2^{-ns}
 \left\|
   2^{sj}\|b_j\|_{\mathcal B}
 \right\|_{\ell^{q_2}_{j\le n}}
 +
 2^{-ns}
 \left\|
   2^{sj}\|c_j\|_{\mathcal B}
 \right\|_{\ell^{q_2}_{j\le n}}.
\]
Then again \eqref{eq:geometric-weight-embedding} implies
\[
 \left\|
   2^{sj}\|b_j\|_{\mathcal B}
 \right\|_{\ell^{q_2}_{j\le n}}
 \le 2^{sn} \left\|
   \|b_j\|_{\mathcal B}
 \right\|_{\ell^{q_1}_{j\le n}}
 \le
 C2^{ns}\|b\|_A.
\]
Since the last term is bounded by $2^{-ns}\|c\|_B$ we conclude that
$L_n(a)<\infty$.

We claim that
\begin{equation}\label{eq:K-estimate-corrected}
 K(2^{-ns},a;A,B)
 \sim
 H_n(a)+L_n(a).
\end{equation}

For the upper bound, define the coordinate truncations
\[
 a^{>n}_j
 :=
 \begin{cases}
  a_j,&j>n,\\
  0,&j\le n,
 \end{cases}
 \qquad
 a^{\le n}_j
 :=
 \begin{cases}
  a_j,&j\le n,\\
  0,&j>n.
 \end{cases}
\]
Since $H_n(a)<\infty$, we have $a^{>n}\in A$, with
\[
 \|a^{>n}\|_A=H_n(a).
\]
Since $L_n(a)<\infty$, we have $a^{\le n}\in B$, with
\[
 2^{-ns}\|a^{\le n}\|_B=L_n(a).
\]
The decomposition
\[
 a=a^{>n}+a^{\le n}
\]
is therefore an admissible decomposition in the definition of the
$K$-functional. Consequently,
\[
 \begin{aligned}
 K(2^{-ns},a;A,B)
 &\le
 \|a^{>n}\|_A
 +
 2^{-ns}\|a^{\le n}\|_B\\
 &=H_n(a)+L_n(a).
 \end{aligned}
\]

For the reverse inequality, let
\[
 a=b+c,
 \qquad b\in A,\quad c\in B.
\]
By the triangle inequality,
\begin{equation}\label{eq:H-decomposition}
 H_n(a)
 \le
 \left\|\|b_j\|_{\mathcal B}\right\|_{\ell^{q_1}_{j>n}}
 +
 \left\|\|c_j\|_{\mathcal B}\right\|_{\ell^{q_1}_{j>n}}.
\end{equation}
The first term is bounded by $\|b\|_A$. For the second term, write
\[
 \|c_j\|_{\mathcal B}
 =
 2^{-ns}2^{-s(j-n)}
 \bigl(2^{sj}\|c_j\|_{\mathcal B}\bigr),
 \qquad j>n.
\]
Applying \eqref{eq:geometric-weight-embedding} with
$\alpha=s$, $p=q_1$, and $q=q_2$, after the change of variables
$m=j-n$, gives
\[
 \begin{aligned}
 \left\|\|c_j\|_{\mathcal B}\right\|_{\ell^{q_1}_{j>n}}
 &=
 2^{-ns}
 \left\|
   2^{-s(j-n)}
   2^{sj}\|c_j\|_{\mathcal B}
 \right\|_{\ell^{q_1}_{j>n}}\\
 &\le
 C2^{-ns}
 \left\|
   2^{sj}\|c_j\|_{\mathcal B}
 \right\|_{\ell^{q_2}_{j>n}}\\
 &\le
 C2^{-ns}\|c\|_B.
 \end{aligned}
\]
Consequently,
\begin{equation}\label{eq:H-lower-bound}
 H_n(a)
 \le
 C\bigl(
   \|b\|_A+2^{-ns}\|c\|_B
 \bigr).
\end{equation}

We estimate $L_n(a)$ similarly. By the triangle inequality,
\[
 \begin{aligned}
 L_n(a)
 &\le
 2^{-ns}
 \left\|
   2^{sj}\|b_j\|_{\mathcal B}
 \right\|_{\ell^{q_2}_{j\le n}}\\
 &\quad+
 2^{-ns}
 \left\|
   2^{sj}\|c_j\|_{\mathcal B}
 \right\|_{\ell^{q_2}_{j\le n}}.
 \end{aligned}
\]
The second term is bounded by $2^{-ns}\|c\|_B$. For the first term,
observe that, for $j\le n$,
\[
 2^{sj}\|b_j\|_{\mathcal B}
 =
 2^{ns}2^{-s(n-j)}\|b_j\|_{\mathcal B}.
\]
Using \eqref{eq:geometric-weight-embedding}, now with the change of
variables $m=n-j$, we obtain
\[
 \begin{aligned}
 \left\|
   2^{sj}\|b_j\|_{\mathcal B}
 \right\|_{\ell^{q_2}_{j\le n}}
 &=
 2^{ns}
 \left\|
   2^{-s(n-j)}\|b_j\|_{\mathcal B}
 \right\|_{\ell^{q_2}_{j\le n}}\\
 &\le
 C2^{ns}
 \left\|
   \|b_j\|_{\mathcal B}
 \right\|_{\ell^{q_1}_{j\le n}}\\
 &\le
 C2^{ns}\|b\|_A.
 \end{aligned}
\]
It follows that
\begin{equation}\label{eq:L-lower-bound}
 L_n(a)
 \le
 C\bigl(
   \|b\|_A+2^{-ns}\|c\|_B
 \bigr).
\end{equation}
Combining \eqref{eq:H-lower-bound} and
\eqref{eq:L-lower-bound}, we find
\[
 H_n(a)+L_n(a)
 \le
 C\bigl(
   \|b\|_A+2^{-ns}\|c\|_B
 \bigr).
\]
Taking the infimum over all decompositions $a=b+c$ gives
\[
 H_n(a)+L_n(a)
 \le
 C K(2^{-ns},a;A,B).
\]
Together with the upper bound, this proves
\eqref{eq:K-estimate-corrected}.

We next discretize the real interpolation norm. Suppose first that
$1\le p<\infty$. By definition,
\[
 \|a\|_{(A,B)_{\theta,p}}
 =
 \left(
  \int_0^\infty
  \bigl(t^{-\theta}K(t,a;A,B)\bigr)^p
  \frac{dt}{t}
 \right)^{1/p}.
\]
Since $s>0$, the intervals
\[
 I_n:=[2^{-(n+1)s},2^{-ns}],
 \qquad n\in\mathbb Z,
\]
form a partition of $(0,\infty)$. The standard monotonicity
properties of the $K$-functional imply that, for $t\in I_n$,
\[
 t^{-\theta}K(t,a;A,B)
 \sim
 2^{n\theta s}K(2^{-ns},a;A,B),
\]
with constants depending only on $s$ and $\theta$. Therefore,
\begin{equation}\label{eq:K-discretization}
 \|a\|_{(A,B)_{\theta,p}}
 \sim
 \left\|
   2^{n\theta s}K(2^{-ns},a;A,B)
 \right\|_{\ell^p_n}.
\end{equation}
When $p=\infty$, the same argument gives
\[
 \|a\|_{(A,B)_{\theta,\infty}}
 \sim
 \sup_{n\in\mathbb Z}
 2^{n\theta s}K(2^{-ns},a;A,B).
\]

Using \eqref{eq:K-estimate-corrected}, we obtain, for every
$1\le p\le\infty$,
\[
 \begin{aligned}
 &\left\|
   2^{n\theta s}K(2^{-ns},a;A,B)
 \right\|_{\ell^p_n}\\
 &\qquad\sim
 \left\|
   2^{n\theta s}H_n(a)
 \right\|_{\ell^p_n}
 +
 \left\|
   2^{n\theta s}L_n(a)
 \right\|_{\ell^p_n}.
 \end{aligned}
\]
Here we used the elementary equivalence
\[
 \|a+b\|_{\ell^p}
 \sim
 \|a\|_{\ell^p}+\|b\|_{\ell^p}
\]
for nonnegative sequences $a$ and $b$, including $p=\infty$.

For the first term, the discrete Hardy inequality
\eqref{eq:hardy-tail}, with $\alpha=\theta s>0$, gives
\[
 \begin{aligned}
 \left\|
   2^{n\theta s}H_n(a)
 \right\|_{\ell^p_n}
 &=
 \left\|
   2^{n\theta s}
   \left\|\|a_j\|_{\mathcal B}\right\|_{\ell^{q_1}_{j>n}}
 \right\|_{\ell^p_n}\\
 &\sim
 \left\|
   2^{\theta sj}\|a_j\|_{\mathcal B}
 \right\|_{\ell^p_j}.
 \end{aligned}
\]

For the second term, define the nonnegative scalar sequence
\[
 d_j:=2^{sj}\|a_j\|_{\mathcal B}.
\]
Then
\[
 2^{n\theta s}L_n(a)
 =
 2^{-n(1-\theta)s}
 \|d_j\|_{\ell^{q_2}_{j\le n}}.
\]
Since $(1-\theta)s>0$, the discrete Hardy inequality
\eqref{eq:hardy-head} yields
\[
 \begin{aligned}
 \left\|
   2^{n\theta s}L_n(a)
 \right\|_{\ell^p_n}
 &=
 \left\|
   2^{-n(1-\theta)s}
   \|d_j\|_{\ell^{q_2}_{j\le n}}
 \right\|_{\ell^p_n}\\
 &\sim
 \left\|
   2^{-(1-\theta)sj}d_j
 \right\|_{\ell^p_j}\\
 &=
 \left\|
   2^{-(1-\theta)sj}
   2^{sj}\|a_j\|_{\mathcal B}
 \right\|_{\ell^p_j}\\
 &=
 \left\|
   2^{\theta sj}\|a_j\|_{\mathcal B}
 \right\|_{\ell^p_j}.
 \end{aligned}
\]
Combining these estimates with \eqref{eq:K-discretization}, we
conclude that
\[
 \|a\|_{(A,B)_{\theta,p}}
 \sim
 \left\|
   2^{\theta sj}\|a_j\|_{\mathcal B}
 \right\|_{\ell^p_j}.
\]
This proves the proposition when $s>0$.

It remains to consider $s<0$. Define the index-reversal operator
$R$ by
\[
 (Ra)_j:=a_{-j}.
\]
For every $1\le q\le\infty$, the operator $R$ is an isometric
isomorphism on $\ell^q(\mathcal B)$. Moreover,
\[
 \begin{aligned}
 \|Ra\|_{\dot\ell^{-s}_q(\mathcal B)}
 &=
 \left\|
   2^{-sj}\|a_{-j}\|_{\mathcal B}
 \right\|_{\ell^q_j}\\
 &=
 \left\|
   2^{sk}\|a_k\|_{\mathcal B}
 \right\|_{\ell^q_k}\\
 &=
 \|a\|_{\dot\ell^s_q(\mathcal B)},
 \end{aligned}
\]
where $k=-j$. Thus $R$ identifies the couple
\[
 \bigl(
   \dot\ell^0_{q_1}(\mathcal B),
   \dot\ell^s_{q_2}(\mathcal B)
 \bigr)
\]
isometrically with
\[
 \bigl(
   \dot\ell^0_{q_1}(\mathcal B),
   \dot\ell^{-s}_{q_2}(\mathcal B)
 \bigr),
\]
where $-s>0$. Applying the already proved positive-exponent case
gives
\[
 \bigl(
   \dot\ell^0_{q_1}(\mathcal B),
   \dot\ell^s_{q_2}(\mathcal B)
 \bigr)_{\theta,p}
 =
 \dot\ell^{\theta s}_p(\mathcal B)
\]
with equivalent norms. This completes the proof.
\end{proof}

From \eqref{equivBesov} and Proposition \ref{prop:weighted-sequence-interpolation} we obtain the following.

\begin{cor}[Real interpolation]\label{trealint}
Let $-1\le s<0$, $1\le r_1\le\infty$, and $0<\theta<1$. Then
\begin{equation}\label{eq:complex-int}
\begin{aligned}
& (\,L^2,\ \mathcal F^{-1}\dot B^{s}_{2,r_1}\,)_{\theta,2}
 \;=\; \mathcal{F}^{-1} (\dot B^{\,\theta s}_{2,2})\\
 &\;=\; \mathcal{F}^{-1}\dot H^{\theta s}
 \;=\; \{ f : |\xi|^{\theta s} f(\xi) \in L^2_\xi \}
\end{aligned}
\end{equation}
with equivalent norms.
\end{cor}

\section{Uniform pseudoconformal bounds}
The mass and energy are conserved:
\begin{equation}
 \|u(t)\|_2^2=\|u_0\|_2^2,
 \qquad E_u(t)=E_u(0),
\end{equation}
where
\begin{equation}\label{energydef}
 E_u(t)=\frac14\int|\nabla u|^2dx+\frac13\int|u|^3dx.
\end{equation}
The pseudoconformal method appears in \cite{GV1980,TY84,W85,CW92}. Define
\begin{equation}\label{eq:EJ-def}
 \mathcal E_u^{\rm conf}(t)
 =\frac14\int|J(t)u|^2dx+\frac{t^2}{3}\int|u|^3dx,
 \qquad J(t)=x+it\nabla.
\end{equation}
For \eqref{NLSmm},
\begin{equation}\label{pcid}
 \frac{d}{dt}\mathcal E_u^{\rm conf}(t)
 =\frac{t}{3}\int|u(t,x)|^3dx.
\end{equation}
\begin{lem}\label{uniform1}
For $u_0\in\Sigma$,
\[
 \mathcal E_u^{\rm conf}(t)\le Ct,\qquad t\ge1.
\]
\end{lem}
\begin{proof}
From \eqref{eq:EJ-def},
\[
 \int|u|^3dx\le\frac3{t^2}\mathcal E_u^{\rm conf}(t).
\]
Substitution into \eqref{pcid} gives
\[
 \frac{d}{dt}\mathcal E_u^{\rm conf}(t)
 \le\frac1t\mathcal E_u^{\rm conf}(t).
\]
Integrating from $1$ to $t$ yields the result.
\end{proof}
\begin{cor}\label{cor:kinPot}
If $u_0\in\Sigma\setminus\{0\}$, then
\[
 \int_{\R^2}|u(t,x)|^3dx\lesssim t^{-1},
 \qquad
 \lim_{t\to\infty}\frac14\|\nabla u(t)\|_2^2=E_u(0)>0.
\]
\end{cor}
\begin{proof}
The first assertion follows from $t^2\int|u|^3\lesssim t$. The second follows from energy conservation and the decay of the potential energy.
\end{proof}

\section{Radiality and the degree-one spherical-harmonic sector}

\begin{prop}\label{prop:radiality}
    Let \(n\ge2\), \(v\in\mathbb R^n\), and let $g:(0,\infty)  \mapsto \mathbb C$ be a function such that \footnote{here $g \in L^2((0,\infty);r^{k-1}dr)$, $k \in \mathbb N,$ means $\int_0^\infty |g(r)|^2 r^{k-1} dr < \infty.$ In this case $g(|x|)$, $x \in \mathbb R^k$ is a radial function in $L^2(\mathbb R^k).$} \(g\in L^2((0,\infty);r^{n+1}\,dr)\). Define
    \[ f(x)=(x\cdot v)g(|x|),\qquad x\in\mathbb R^n. \]
    Then \(f\in L^2(\mathbb R^n)\). Moreover, for every bounded Borel function \(\phi\),
    \begin{equation}\label{eq.idD0}
 \phi(-\Delta_{\mathbb R^n})f(x) = (x\cdot v)\, \phi(-\Delta_{\mathrm{rad}}^{(n+2)})g(|x|)
\end{equation}
    in \(L^2(\mathbb R^n)\), where
    where \(-\Delta_{\mathrm{rad}}^{(n+2)}\) denotes the nonnegative self-adjoint radial Laplacian in dimension \(n+2\), formally given by \[ \Delta_{\mathrm{rad}}^{(n+2)}g(r) = g''(r)+\frac{n+1}{r}g'(r). \]
\end{prop}

\begin{proof} First,
\[ \|f\|_{L^2(\mathbb R^n)}^2 = \int_0^\infty |g(r)|^2 r^{n+1}\,dr \int_{\mathbb S^{n-1}} |\omega\cdot v|^2\,d\omega. \] Since
\[ \int_{\mathbb S^{n-1}} |\omega\cdot v|^2\,d\omega = \frac{|\mathbb S^{n-1}|}{n}|v|^2, \]
we have \(f\in L^2(\mathbb R^n)\). If \(v=0\), the result is trivial.
Assume \(v\neq0\), and put \(e=v/|v|\). Let
\[ Y_e(\omega) = \left(\frac{n}{|\mathbb S^{n-1}|}\right)^{1/2}\omega\cdot e. \]
Then \(Y_e\) is a normalized spherical harmonic of degree one and
\[ \Delta_{\mathbb S^{n-1}}Y_e=-(n-1)Y_e. \]
Define
\[ U:L^2((0,\infty);r^{n+1}\,dr)\to L^2(\mathbb R^n) \]
by
\[ (Ug)(r\omega)=r g(r)Y_e(\omega). \]
Then \(U\) is unitary from \(L^2((0,\infty);r^{n+1}\,dr)\) onto the closed subspace
\[ \mathcal H_e = \{h(r)Y_e(\omega): h(r)=r g(r),\ g\in L^2((0,\infty);r^{n+1}\,dr)\}. \]
For smooth and compactly supported \(g\), write \(h(r)=r g(r)\). Since
\[ \Delta_{\mathbb R^n}(hY_e) = \left( h''+\frac{n-1}{r}h' -\frac{n-1}{r^2}h \right)Y_e, \]
we obtain \[ h'=g+rg', \qquad h''=2g'+rg''. \]
Therefore
\[ h''+\frac{n-1}{r}h'-\frac{n-1}{r^2}h = rg''+(n+1)g'. \]
Hence
\[ \Delta_{\mathbb R^n}(Ug) = r\left(g''+\frac{n+1}{r}g'\right)Y_e = U\bigl(\Delta_{\mathrm{rad}}^{(n+2)}g\bigr). \] Equivalently,
\[ -\Delta_{\mathbb R^n}Ug = U\bigl(-\Delta_{\mathrm{rad}}^{(n+2)}g\bigr) \]
on \(C_c^\infty(0,\infty)\). The degree-one spherical-harmonic sector is reducing for \(-\Delta_{\mathbb R^n}\). Since the above intertwining holds on a core for \(-\Delta_{\mathrm{rad}}^{(n+2)}\), the restriction of \(-\Delta_{\mathbb R^n}\) to \(\mathcal H_e\) is unitarily equivalent to \(-\Delta_{\mathrm{rad}}^{(n+2)}\). Thus, by the spectral theorem, for every bounded Borel function \(\phi\),
\[ \phi(-\Delta_{\mathbb R^n})Ug = U\phi(-\Delta_{\mathrm{rad}}^{(n+2)})g. \]
Finally,
\[ f(x) = (x\cdot v)g(|x|) = |v|\left(\frac{|\mathbb S^{n-1}|}{n}\right)^{1/2}Ug(x). \]
Multiplying the previous identity by this constant gives
\[ \phi(-\Delta_{\mathbb R^n})f(x) = (x\cdot v) \bigl(\phi(-\Delta_{\mathrm{rad}}^{(n+2)})g\bigr)(|x|) \] in \(L^2(\mathbb R^n)\).
\end{proof}

\begin{cor}\label{cor:propagator-sector}
Under the assumptions of Proposition \ref{prop:radiality},
\begin{equation}\label{eq.idD1}
 e^{it\Delta/(2m)}f(x)
 =(x\cdot v)e^{it\Delta^{(n+2)}/(2m)}g(|x|).
\end{equation}
\end{cor}
\begin{proof}
Apply Proposition \ref{prop:radiality} to the bounded Borel function $\phi(\lambda)=e^{-it\lambda/(2m)}$.
\end{proof}

\begin{cor}\label{corD2}
Let $g:(0,\infty) \ni  r \to \mathbb C$ be a function such that  $rg\in L^2_{\rm rad}(\R^n)$. For $j=1,\ldots,n$,
\[
 e^{it\Delta/(2m)}[x_jg(|x|)]
 =x_j e^{it\Delta^{(n+2)}/(2m)}g(|x|).
\]
\end{cor}

\section*{Acknowledgements}
The authors are grateful to N. Hayashi for discussions during the preparation of the manuscript.
The first author was supported in part by
  Gruppo Nazionale per l'Analisi Matematica,  by Institute of Mathematics and Informatics, Bulgarian Academy of Sciences and by  Waseda University. The second author was supported by JSPS KAKENHI Grant Number 24H00024 and JST Moonshot R \& D Grant number GMPJMS25A5.

\bibliographystyle{amsplain}
\bibliography{MWN_OG}
\end{document}